\documentclass[11pt]{article}
\usepackage{amsmath}
\usepackage{amssymb}
\usepackage{amsthm}
\usepackage{mathrsfs}
\usepackage{mathtools}
\usepackage{graphicx}
\usepackage[a4paper, total={6.5in, 9in}]{geometry}
\usepackage{setspace}
\usepackage{tikz}
\usepackage{array}
\usepackage{makecell}
\usepackage{longtable}
\usepackage[utf8]{inputenc}

\DeclareMathOperator{\lcm}{lcm}

\title{Senior Thesis - Equal Coverings}
\author{Andrew Velasquez-Berroteran}
\date{\today}

\begin{document}

\DeclarePairedDelimiter\ceil{\lceil}{\rceil}
\DeclarePairedDelimiter\floor{\lfloor}{\rfloor}
\newtheorem{definition}{Definition}
\newtheorem{proposition}{Proposition}
\newtheorem{lemma}{Lemma}
\newtheorem{corollary}{Corollary}
\newtheorem{example}{Example}
\newtheorem{theorem}{Theorem}
\newtheorem{note}{Note}
\newtheorem{conjecture}{Conjecture}
\newtheorem{remark}{Remark}
\onehalfspacing
\begin{titlepage}

\newcommand{\HRule}{\rule{\linewidth}{0.5mm}} % Defines a new command for the horizontal lines, change thickness here

\center % Center everything on the page
 
%----------------------------------------------------------------------------------------
%	HEADING SECTIONS
%----------------------------------------------------------------------------------------

\textsc{\LARGE Department of Mathematics \& Computer Science}\\[1.5cm] % Name of your university/college
 % Include a department/university logo - this will require the graphicx package

%----------------------------------------------------------------------------------------
%	TITLE SECTION
%----------------------------------------------------------------------------------------

\HRule \\[0.4cm]
{ \huge \bfseries Equal Coverings of Finite Groups}\\[0.1cm] % Title of your document
\HRule \\[2cm]
 
%----------------------------------------------------------------------------------------
%	AUTHOR SECTION
%----------------------------------------------------------------------------------------

\begin{minipage}{0.5\textwidth}
\begin{flushleft} \large
\emph{Author:}\\
\textsc{Andrew Velasquez-Berroteran}\\\vspace{20pt} % Your name
\emph{Committee Members:}\\
\textsc{Tuval Foguel (advisor)}\\
\textsc{Joshua Hiller}\\
\textsc{Salvatore Petrilli}\\
\end{flushleft}

\end{minipage}\\[1cm]

% If you don't want a supervisor, uncomment the two lines below and remove the section above
%\Large \emph{Author:}\\
%John \textsc{Smith}\\[3cm] % Your name

%----------------------------------------------------------------------------------------
%	DATE SECTION
%----------------------------------------------------------------------------------------
{\large April 27th, 2022}\\[2cm] % Date, change the \today to a set date if you want to be precise

\vfill % Fill the rest of the page with whitespace

\end{titlepage}
\tableofcontents
\newpage
\begin{abstract}
In this thesis, we will explore the nature of when certain finite groups have an equal covering, and when finite groups do not. Not to be confused with the concept of a cover group, a covering of a group is a collection of proper subgroups whose set-theoretic union is the original group. We will discuss the history of what has been researched in the topic of coverings, and as well as mention some findings in concepts related to equal coverings such as that of equal partition of a group. We develop some useful theorems that will aid us in determining whether a finite group has an equal covering or not. In addition, for when a theorem may not be entirely useful to examine a certain group we will turn to using \texttt{\texttt{GAP}} (Groups, Algorithms, Programming) for computational arguments.

\end{abstract}
\textbf{Motivation}\vspace{5pt}\\
The question of determining how a group may possess an equal covering is an interesting since in addition to wondering if a group can be the set-theoretic union of some of its proper subgroups, we would also like to see if there is a such a collection with all member being the same size. As we will see soon, non-cyclic groups all possess some covering. If we add, however, the restriction mentioned above then the problem of determining such groups becomes a lot more complicated. We hope to determine from a selection of finite groups, which ones have an equal covering and which do not. Our plan will first proceed with familiarizing ourselves with useful definitions, such as that of the exponent of a group. Next, we will mention general research within the topic of coverings in hopes some finding from within the past century may serve us. Afterwards, we will derive our own theorems related to equal coverings of groups. Following that, we will then utilize the theorems presented, as well as \texttt{GAP} for when the theorems alone do not help, in aiding us to determine which groups up to order 60  and some finite (non-cyclic) simple groups have equal coverings.
\section{Introduction}
The topic of coverings of groups is a relatively novel one, only having been researched within the past 120 years. Equal coverings, on the other hand, has not been researched as much and will be the focus of this paper. Given a group $G$ and if $\Pi$ is a a covering of $G$, then it is an equal covering of $G$ if for all $H,K \in \Pi$, we have $H$ and $K$ are of the same order. Now, one thing that must be clear is that not every group will have a covering, let alone an equal covering. In other words, when we know that $G$ has no covering at all, then it is not worthwhile attempting to find an equal covering or determine if it has one or not. To begin this discussion, we will first take notice of a very important fact that distinguishes groups that have coverings, from those that do not. From this point on, unless otherwise specified, we will be concerned with finite coverings of groups, or coverings that have finitely many proper subgroups of the original group.\vspace{5pt}\\
If $G$ is a group, let $\sigma(G)$ denote the smallest cardinality of any covering of $G$. If $G$ has no  covering, then we would simply write $\sigma(G) = \infty$. Below is a relatively simple but powerful well-known theorem.
\begin{theorem}[\cite{scorza}]\label{Cyclic}
Let $G$ be a group. $G$ has a covering if and only if $G$ is non-cyclic.
\end{theorem}
\begin{proof}
Suppose $G$ has an covering. By definition, this is a collection of proper subgroups, where each element of $G$ must appear in at least one of the subgroups. It $x \in G$, then $\langle x \rangle$ must be a proper subgroup of $G$, so $G$ cannot be generated by $x$. Hence, $G$ is non-cyclic.\vspace{5pt}\\
Conversely, suppose $G$ is non-cyclic. Consider the collection of subgroups $\Pi = \{ \langle a \rangle: a \in G\}$. Since $G$ is non-cyclic, $\langle a  \rangle$ is a proper subgroup of $G$ for all $a \in G$, so $\Pi$ is a covering of $G$.
\end{proof}
\noindent A consequence of Theorem \ref{Cyclic} is that all groups of prime order do not have a covering, since all groups of prime order are cyclic. Since this means we will not take much interest in cyclic groups we have limited the number of groups to analyze for having an equal covering, even if the proportion of groups are reduced by very little.\vspace{5pt}\\
In this investigation, we will work primarily with finite groups. Say if $G$ is a finite non-cyclic group, would there be a way to determine $\sigma(G)$, or at the very least find bounds on $\sigma(G)$? In a moment we will look at what has been researched in domain of coverings of groups, which will involve some work in answering this question for some groups. But before we do that, we will mention and prove two well-known theorems related to this question.
\begin{theorem}\label{Union2}
Let $G$ be a non-cyclic group. If $H$ and $K$ are proper subgroups of $G$, then $G$ cannot be the union of $H$ and $K$. In other words, $\sigma(G) \neq 2$ for any non-cyclic group $G$.
\end{theorem}
\begin{proof}
Suppose $H$ and $K$ are proper subgroups such that $G = H \cup K$. Since it cannot be possible for either $H \subseteq K$ or $K \subseteq H$, we must have there is some $h \in H$  but $h \notin K$, and there is some $k \in K$ but $k \notin H$. Since $hk \in G$, $hk \in H$ or $hk \in K$. Observe if $hk \in H$, then since $h^{-1} \in H$, we have $h^{-1}(hk) = (h^{-1}h)k = k \in H$, which is impossible. Similarly, if $hk \in K$ then $(hk)k^{-1} = h(kk^{-1}) = h \in K$. We have a contradiction, so we cannot have $G$ cannot be the union of $H$ and $K$.
\end{proof}
\begin{proposition}\label{Bounds}
If $G$ be a non-cyclic group of order $n$, then $2 < \sigma(G) \leq n - 1$.
\end{proposition}
\begin{proof}
Suppose $G$ is a non-cyclic group of order $n$. Clearly no covering cannot consist of one element, since that would indicate it contains $G$, not a possibility. Next, by Theorem \ref{Union2}, any covering must have more than two proper subgroups of $G$. So, $\sigma(G) > 2$.\\
Now, let $a_1$, $a_2$, ..., $a_{n-1}$ represent all $n-1$ nonidentity elements of $G$. Since $G$ is non-cyclic, $\langle a_i \rangle < G$ for $1 \leq i \leq n-1$. If $\Pi = \{\langle a_i \rangle:\ 1 \leq i \leq n-1\}$, then $\Pi$ is a collection of proper $n-1$ subgroups of $G$. Furthermore, the union of all these subgroups is $G$, so $\Pi$ is a covering of $G$. It follows $\sigma(G) \leq n-1$.  Therefore, $2 < \sigma(G) \leq n-1$.
\end{proof}
We consider Proposition 1 above just a proposition and not a theorem since, as we will see in the history section, there has been work done to find a smaller range for $\sigma(G)$ for different finite groups $G$ as well as specific values for certain groups.\vspace{5pt}\\
As mentioned before, we will only discuss finite groups in this peper, but as a brief mention the possibility of infinite groups being a union of proper subgroups is a bit mystifying. In regards to Theorem \ref{Cyclic}, there is a reason we needed to state beforehand that the groups we refer to will need to be finite. Take for example the group $\mathbb{Q}^{+}$ under multiplication. While this group may not be cyclic, Haber and Rosenfeld \cite{haber1959groups} demonstrated that it's actually impossible for $\mathbb{Q}^+$ be a union of proper subgroups. So in addition to the overall complexity that comes with dealing with infinite groups, there will be theorems presented in this thesis that may not hold true for infinite groups satisfying the necessary assumptions.
\section{History}
\subsection*{On the General History of  Group Coverings}
\indent Before we continue with our discussion talking about equal coverings, let's take a look at some things that have been researched within the topic of coverings of groups, as well as a mention on coverings of loops and equal partitions.\vspace{5pt}\\
\indent The first instance of there being a discussion of representing groups as a general union of proper subgroups appeared in a book from G. Scorza in 1926. Two decades prior, G.A. Miller had actually touched on the concept of partitions which we will dedicate its own subsection to later in this section. Although this was the first instance wherein a mathematician posed a problem relevant to the idea of coverings for groups, one source of great motivation for inquiry came from P. Erdös.\vspace{5pt}\\
\indent Erdös is said to be a very influential mathematician, with some arguing he is the most prolific one from the last century. He had done extensive work in various fields of mathematics, especially in the realm in algebra. Scorza had originally come up with the idea of coverings for groups in the 1920s, and in a matter of less than half a century later, Erdös posed somewhat of a related question. The question can ultimately be boiled down to the following \cite{neumann_1976}:\\
If $G$ is a group and there is no infinite subset of elements which do not commute, is there a finite number of such subsets? \\
While Erdös was essentially talking of coverings for groups, but by particular subsets and not proper subgroups, his question helped mathematicians such as B.H Neumann looked at groups with this property, and some other mathematicians such as H.E. Bell and L.C. Kappe look at a ring theory problem analogous to Erdös' \cite{bell1997analogue}. Thus we definitely say Erdös served to help bring attention to the theory of coverings of groups, which Neumann and Kappe both looked more into as we will see later in this section.\vspace{5pt}\\
\indent There was some work already done within this topic even prior to Erdös' involvement, so we will continue on from the relatively early twentieth century. Theorem \ref{Union2} has showed us it's impossible to write a group as union of two proper subgroups, but it is possible for a group to be a union of three of its proper subgroups and as it turns out, there's a theorem for this. This theorem and Theorem \ref{Cyclic} have repeatedly been mentioned and proven in multiple papers such as in \cite{haber1959groups} and \cite{bruckheimer}, but first appeared in Scorza's paper \cite{scorza}.
\begin{theorem}[\cite{scorza}]
If $G$ is a group, then $\sigma(G) = 3$ if and only if for some $N \vartriangleleft G$, $G/N \cong V$, the Klein 4-group.
\end{theorem}
An immediate consequence of this theorem is that the lower bound of the inequality given in Theorem \ref{Bounds} can be changed to 3 and so now for any finite non-cyclic group $G$
we have $3 \leq \sigma(G) < n-1$. Immediately we see that smallest non-cyclic group that has a covering is indeed $V$ and it should be evident that $\{\langle(0,1)\rangle, \langle (1,0)\rangle, \langle (1,1)\rangle\}$ forms a covering of $V$. In fact, it happens to be an equal covering of $V$.
\begin{definition}
Given a group $G$ and a covering $\Pi = \{H_1, H_2 ,..., H_n\}$, we say $\Pi$ is \textbf{irredundant}( or \textbf{minimal}) if for any $H_i \in \Pi$, $H_i$ is not contained in the union of the remaining $H's$ in $\Pi$. In other words, for each $i \in \{1,..., n\}$ there exists $x_i \in H_i$ such that $x_i \notin \bigcup\limits_{j\neq i}H_j$.
\end{definition}
Ideally when we come up with a covering for a group, we want the least amount of subgroups necessary. \cite{haber1959groups} actually had proven that if $\Pi = \{H_i\}$ is an irredundant covering of $G$ then for any $H_i \in \Pi$, $H_i$ contains the intersection of the remaining $H's$ in $\Pi$. Further in their paper they had shown the following two statements for any finite group $G$:
\begin{theorem}[\cite{haber1959groups}]\label{haber}
(i) If $p$ is the smallest prime divisor of $|G|$ then $G$ cannot be the union of $p$ or fewer proper subgroups.\\
(ii) If $p$ is the smallest prime divisor of $|G|$ and  $\Pi = \{H_i\}$ is a covering of $p+1$ proper subgroups, there is some $H_i$ for which $[G:H_i] = p$. If such an $H_i$ is normal, then all $H's 
\in \Pi$ have index $p$ and $p^2$ divides $|G|$.
\end{theorem}
As mentioned, Theorem 4 has been repeatedly mentioned in multiple papers and in M. Bruckheimer, et. al \cite{bruckheimer}, they had actually explored a little more of when groups can be the union of three proper subgroups. As an example, they had explained all dihedral groups of orders that are divisible by 4 and all dicyclic groups are `3-groups', which in the context of their paper means their covering number is 3. Additionally, they had shown if a group $G$ has the decomposition (or covering) of $\{A,B,C\}$ then this is only possible if all three subgroups are abelian, all are non-abelian, or only one is abelian. They had shown it was impossible for a covering of $G$ to have 2 abelian subgroups of $G$ and 1 non-abelian.\vspace{5pt}\\
\indent T. Foguel and M. Ragland \cite{foguel2008groups} actually investigate what they call `CIA'-groups, or groups that have a covering whose components are isomorphic abelian subgroups of $G$. They had found many results such as that every finite group can be a factor of a CIA-group, and that the (direct) product of two CIA-groups is a CIA-group. Among the other results they had derived, they had found which families of groups are CIA-groups and which ones do not. All dihedral groups and groups of square-free order are examples of non-CIA-groups and generally any non-cyclic group with prime exponent is a CIA-group. Since isomorphic groups have the same order, any finite CIA-group by definition will have an equal covering, or covering by proper subgroups of the same order.\vspace{5pt}\\
\indent J.H.E. Cohn \cite{cohn1994n} provide us with plenty of nifty theorems and corollaries. Before presenting two superb theorems from his paper we must mention that in place of\ $\bigcup$, Cohn used summation notation and so if $\{H_1, H_2, ..., H_n\}$ is a covering for $G$, with $|H_1| \geq |H_2| \geq ... |H_n|$, then he had written $G = \sum\limits_{i=1}^{n}H_i$. He had also used $i_r$ to denote $[G:H_r]$ and if $\sigma(G) = n$ he said that $G$ is an $n$-sum group.
\begin{theorem}[\cite{cohn1994n}]\label{cohn1}
Let $G$ be a finite $n$-sum group. It follows:
\begin{enumerate}
    \item $i_2 \leq n-1$
    \item if $N \vartriangleleft G$ then $\sigma(G) \leq \sigma(G/N)$
    \item $\sigma(H \times K) \leq \min\{\sigma(H), \sigma(K)\}$, where equality holds if and only if $|H|$ and $|K|$ are coprime.
\end{enumerate}
\end{theorem}
Before we continue, we must mention that Theorem \ref{cohn2} was originally written so that \textit{1.} and \textit{2.} were lemmas and \textit{3.} was an immediate corollary. In our study of equal coverings, any one of these may prove to be useful so we compiled all three statements into a theorem. Before we move on to the next theorem, we must note that Cohn defined a primitive $n$-sum group $G$ to be a group such that $\sigma(G) = n$ and $\sigma(G/N) > n$ for all nontrivial normal subgroups $N$ of $G$. The following theorem was written by \cite{bhargava2009groups} with \textit{2.}-\textit{4.} coming originally from Theorem 5 of \cite{cohn1994n} and \textit{5.} coming from work developed later on in the same paper.
\begin{theorem}[\cite{cohn1994n}, \cite{tomkinson}]\label{cohn2}
\vspace{5pt}
\begin{enumerate}
    \item There are no 2-sum groups.
    \item $G$ is a 3-sum group if and only if it has at least two subgroups of index 2. The only primitive 2-sum group is $V$.
    \item $G$ is a 4-sum group if and only if $\sigma(G) \neq 3$ and it has at least 3 subgroups of index 3. The only primitive 4-sum groups are $\mathbb{Z}_3^2$ and $S_3$.
    \item $G$ is a 5-sum group if and only if $\sigma(G) \neq 3$ or 4 and it has at least one maximal subgroup of index 4. The only primitive 5-sum group is $A_4$.
    \item $G$ is a 6-sum group if and only if $\sigma(G) \neq 3$, 4, or 5 and there is a quotient isomorphic to $\mathbb{Z}_5^2$, $D_{10}$ (dihedral group of order 10) or $W = \mathbb{Z}_5 \rtimes \mathbb{Z}_4 = \langle a,b|\ a^5 = b^4 = e, ba = a^2b\rangle$. All three happen to be the only primitive 6-sum groups.
    \item There are no 7-sum groups, or no $G$ for which $\sigma(G) = 7$.
\end{enumerate}
\end{theorem}
\noindent The last statement from Theorem \ref{cohn2}  is interesting since it is the third positive integer for which no groups can be covered by that number of proper subgroups, and although Cohn didn't know or demonstrate a proof of it, it was ultimately proven by M.J. Tomkinson \cite{tomkinson}. In M. Garonzi et. al.'s paper \cite{garonzi2019integers}, one topic of the paper was to figure out what are some integers that cannot be covering numbers. For a complete list of integers less than 129 that cannot be covering numbers, please see \cite{garonzi2019integers}. In particular, they had found that integers which can be covering numbers are of the form $\frac{q^m-1}{q-1}$, where $q$ is a prime and $m \neq 3$. Additionally, something Cohn had also conjectured, and was then proven by Tomkinson, was that for every prime number $p$ and positive integer $n$ there exists a group $G$ for which $\sigma(G) = p^n + 1$, and moreover, such groups are non-cyclic solvable groups.\vspace{5pt}\\
\indent In addition to determining what integers smaller than 129 cannot be a covering number, \cite{garonzi2019integers} also attempted to look at covering numbers of small symmetric groups, linear groups, and some sporadic groups. Some of the results were based on the work of A. Maroti \cite{maroti2005covering}, with one result being that that for all odd $n \geq 3$, except $n =9$, $\sigma(S_n) = 2^{n-1}$. \cite{kappe2016covering} had actually demonstrated that $\sigma(S_9) = 256$, so that formula actually holds for all odd integers greater than 1. Additionally, when finding the exact covering number of a group wasn't available they would at find a lower bound, upper bound or possibly both, such as for Janko group $J_1$, they had found that $5316 \leq \sigma(J_1) \leq 5413$.
\subsection*{Other Types of Coverings}
Now, we have primarily talked thus far groups that have a covering by general proper subgroups. One may ask what if we place restrictions or modify the concept of a standard covering of a group with say a covering by proper normal subgroups, or a covering by proper subgroups with the restriction that any two given subgroups intersect trivially?
\subsubsection*{Covering by Cosets}
 Neumann \cite{neumann1954groups} was interested in seeing what we can find out about when groups can be the union of cosets of subgroups. In other words, he was interested in when $G = \bigcup x_iH_i$. A powerful theorem he had proven was that:
\begin{theorem}[\cite{neumann1954groups}]
If $G = \bigcup x_iH_i$ is a union of cosets of subgroups, and if we remove any $x_iH_i$ for which $[G:H_i]$ is infinite then the remaining union is still all of $G$.
\end{theorem}
\noindent If $G$ is a finite group the Theorem 8 will hold no matter which nontrivial subgroups $H_i$ we choose, but if we were dealing with infinite groups then this theorem can very well prove to incredibly useful.
\subsubsection*{Covering by Normal Subgroups and Conjugates of Subgroups}
M. Bhargava \cite{bhargava2009groups} investigated coverings by normal subgroups and conjugates of subgroups. One type of covering was that of covering by normal subgroups. It was proven that any group that is can be covered by three proper subgroups is actually covered by three normal proper subgroups. Additionally, $G$ can be written as the union of proper normal subgroups of $G$ if and only if there is some quotient group isomorphic to $\mathbb{Z}_{p}^2 = \mathbb{Z}_p \times \mathbb{Z}_p$ for some prime $p$.\\
Another type of covering is that of by conjugate subgroups. It turns out that there isn't an example of a finite group  that is coverable by the conjugates of a single proper subgroup! In \cite{bhargava2009groups} there happens to be a theorem in regard to non-cyclic solvable groups.
\begin{theorem}[\cite{bhargava2009groups}]
Suppose $G$ is a finite non-cyclic solvable group. Then $G$ satisfies either 1) a union of proper normal subgroups or 2) a union of conjugates of 2 proper subgroups. 
\end{theorem}
\noindent Interestingly enough, the infinite group GL$_2(\mathbb{C})$, or group of all non-singular $2 \times 2$ matrices with complex entries, happens to be coverable by the set of all conjugates of upper triangular matrices \cite{bhargava2009groups}.
\subsubsection*{Partitions \& Semi-Partitions}
Now regardless of what type of group covering we have, we only require that such a collection is indeed a covering for the parent group. We now introduce a special kind of covering for groups.\vspace{5pt}\\
As mentioned prior, G.A. Miller \cite{miller1906groups} began an investigation into a special type of covering known as a partition and the purpose of this section is to highlight the many discoveries of partitionable groups.
\begin{definition}
Let $G$ be a group. If $\Pi$ is a covering of $G$ where any two distinct members of $\Pi$ intersect trivially, then $\Pi$ is a \textbf{partition} of $G$. We will say $G$ is partitionable if $G$ has a partition.
\end{definition}
\noindent First, \cite{miller1906groups} had shown two impressive statements: that any abelian partitionable group must be an elementary abelian $p$-group with order $\geq p^2$; and that if $|G| = p^m$ and $\Pi$ is a partition of $G$ then for any $H \in \Pi$ we have $|H| = p^a$ where $a$ divides $m$.\vspace{5pt}\\
Similar to how we defined the covering number of a group, we define $\rho(G)$ to be smallest number of members for any partition of $G$. If $G$ has no partition, then we write $\rho(G) = \infty$. Clearly when $G$ is partitionable, $\sigma(G) \leq \rho(G)$ and so a question may arise as to which groups may satisfy $\sigma(G) < \rho(G)$ and when $\sigma(G) = \rho(G)$. T. Foguel and N. Sizemore \cite{sizemorepartition} look at partition numbers of some finite solvable groups, such as $D_{2n}$ (the dihedral group of order $2n$) and $E_{p^n} = \mathbb{Z}_{p}^n$ (the elementary $p$-abelian group of order $p^n$, where $p$ is prime). In this paper, they mentioned and proven many results, such as when $n > 1$ we have $\rho(E_{p^n}) = 1 + p^{\lceil \frac{n}{2} \rceil}$, as well as that $\sigma(D_{2n}) = \rho(D_{2n})$ if and only if $n$ is prime, otherwise $\sigma(D_{2n}) < \rho(D_{2n})$. During the middle of the last century, work  has been do to classify all partitionable groups, and such a classification was finally complete in 1961 and is due to the work of R. Baer \cite{baer1961partitionen}, O. Kegel \cite{kegel1961nicht}, M. Suzuki \cite{suzuki1961finite} collectively. \vspace{5pt}\\
Let us familiarize ourselves with notation that will be used for the following theorem. If $G$ is a $p$-group, then we define $H_p(G) = \langle x \in G:\ x^p \neq 1\}$ and a group is of Hughes-Thompson type if $G$ is a non-$p$-group where $H_p(G) \neq G$. For the classification mentioned above, please see Theorem 10.
\begin{theorem}[\cite{baer1961partitionen}, \cite{kegel1961nicht}, \cite{suzuki1961finite}]
$G$ is a partitionable group if and only if $G$ is isomorphic to any of the following:
\begin{enumerate}
    \item $S_4$
    \item A $p$-group where $|G| > p$ and $H_p(G) < G$
    \item A Frobenius group ($G = H \rtimes K$, where $H$ is the Frobenius complement and $K$ is the Frobenius kernel)
    \item A group of Hughes-Thompson type
    \item $\text{PSL}(2, p^n)$, $p$ is prime and $p^n \geq 4$
    \item $\text{PGL}(2, p^n)$, $p$ is an odd prime and $p^n \geq 5$
    \item $\text{Sz}(q)$, the Suzuki group of order $q^2(q^2+1)/(q-1)$ where $q = 2^{2n+1}, n\geq 1$
\end{enumerate}
\end{theorem}
After this work, G. Zappa \cite{zappa2003partitions}
had developed a more general concept of partitions, strict $S$-partitions.
\begin{definition}
If $G$ is a group and $\Pi$ is a partition of $G$ such that for all $H_i \cap H_j = S$ for all $H_i, H_j \in \Pi$ and for some $S < G$, then we say $\Pi$ is a \textbf{strict $S$-partition}. If, in addition, $|H_i| = |H_j|$ for all $H_i,H_j \in \Pi$ then we say $\Pi$ is an \textbf{equal strict $S$-partition} or an \textbf{$ES$-partition}.
\end{definition}
One powerful derivation of G. Zappa's was that if $N \leq S < G$  and $N \vartriangleleft G$ then $G$ has a strict $S$-partition $\{H_1, H_2, ..., H_n\}$ if and only if $\{H_1/N, H_2/N,..., H_n/N\}$ is a strict $S/N$-partition of $G/N$.\vspace{5pt}\\
Using Zappa's results and definitions, L. Taghvasani and M. Zarrin \cite{jafari2018criteria} proved among many results that a group $G$ is nilpotent if and only if for every subgroup $H$ of $G$, there is some $S \leq H$ such that $H$ has an $ES$-partition.\vspace{5pt}\\
In 1973, I.M. Isaacs \cite{isaacs1973equally} attempted to look at groups that were equally partitionable, or using Zappa's terminology, all $G$ that have $E\{1\}$-partition. He derived the following theorem:
\begin{theorem}[\cite{isaacs1973equally}]\label{isaacstheorem} $G$ is a finite group with equal partition if and only if $G$ is a finite non-cyclic $p$-group with exponent $p$ where $p$ is a prime.
\end{theorem}
\noindent Isaac's result provides us an insight into at least one class of groups that have equal coverings, since an equal partition is an equal covering after all.\vspace{5pt}\\
\indent To close this subsection, we will talk briefly about \textit{semi-partitions} of groups, which are coverings of groups wherein the intersection of any three distinct components is trivial. Foguel et. al. \cite{semi-partitions} analyze and look for properties of groups that have or do not possess a semi-partition, as well as determine the semi-partition number of a group, $\rho_s(G)$. Some results they had found included that if $G$ has a semi-partition composed of proper normal subgroups, then $G$ is finite and solvable (\cite{semi-partitions}, Theorem 2.1) and when $p$ is prime we have $\sigma(D_{2p^n}) = p + 1$, $\rho(D_{2p^n}) = p^n + 1$, and $\rho_s(D_{2p^n}) = p^n - p^{n-1} + 2$ (\cite{semi-partitions}, Proposition 4.2).
\subsubsection*{Coverings of Loops}
This last subsection on the history of coverings of groups is dedicated to looking over coverings of loops. Indeed, the concept of coverings of groups can be loosely be translated to that of other algebraic structures such as loops, semigroups \cite{kappe2001analogue}, and rings \cite{bell1997analogue}. We will however focus on loops covered by subloops and even subgroups, as well as a brief mention of loop partitions.\vspace{5pt}\\
Similar to how we defined a group covering, T. Foguel and L.C. Kappe \cite{foguel2005loops} define a subloop covering of a loop $\mathscr{L}$ to be a collection of proper subloops $\mathscr{H}_1,..., \mathscr{H}_n$ whose set-theoretic union is $\mathscr{L}$. Using the terminology they had used, $\mathscr{L}$ is \textit{power-associative} if the subloop generated by $x$ forms a group for any $x \in \mathscr{L}$, and \textit{diassociative} if the subloop generated by  $x$ and $y$ form a group for any $x,y \in \mathscr{L}$.\\
Foguel and Kappe then defined the concept of an \textit{$n$-covering} for a loop. We say the collection of proper subloops $\{\mathscr{H}_i: i \in \Omega\}$ is an $n$-covering for  $\mathscr{L}$ if for any collection of $n$ elements of $\mathscr{L}$, those elements lie in $\mathscr{H}_i$ for some $i \in \Omega$. Using this definition, they had proven the following theorem.
\begin{theorem}[\cite{foguel2005loops}]
Given a loop $\mathscr{L}$ we have 
\begin{enumerate}
    \item $\mathscr{L}$ has a 1-covering (or just covering) if and only if $\mathscr{L}$ is power-associative
    \item $\mathscr{L}$ has a 2-covering if and only if $\mathscr{L}$ is diassociative
    \item $\mathscr{L}$ has a 3-covering if and only if $\mathscr{L}$ is a group
\end{enumerate}
\end{theorem}
\noindent In the same paper, Foguel and Kappe that while a few ideas and properties of group coverings can be translated when talking about loops, in other instances we would need to place restrictions in in order to obtain results or theorems analogous to the theorems of group coverings. Theorem 6.4 of \cite{foguel2005loops} we would say is almost the loop equivalent of Theorem 8 of this paper, which was originally derived by B.H. Neumann.\vspace{5pt}\\
In a separate paper, T. Foguel and  R. Atanasov \cite{atanasov2014loops} go further with investigating the subject of loop partitions, which of course can be defined similar to how we define group partitions. First, a \textit{group covering} of loop $\mathscr{L}$ is a covering of subloops that also are subgroups. A group covering is a group-partition (or $G$-partition) if every nonidentity element lies in one subgroup of the covering, and is an equal group partition (or $EG$-partition) if such subgroups are of the same order. T. Foguel and R. Atanasov proved many results using these definitions with one being of being super interest for this paper:
\begin{theorem}[\cite{atanasov2014loops}]
If $\mathscr{L}$ is a finite non-cyclic power-associative loop with the propery $(ab)^n = a^nb^n$ for all $a,b \in \mathbb{N}$, then the following are equivalent:
\begin{enumerate}
    \item $\mathscr{L}$ has a proper $G$-partition
    \item $\mathscr{L}$ has a proper diassociative partition
    \item $\mathscr{L}$ has exponent $p$, where $p$ is prime
\end{enumerate}
\end{theorem}
\noindent Foguel and Atansov also demonstrate that for a certain type of finite non-cyclic loops they have an $EG$-partition if and only if they have prime exponent (\cite{atanasov2014loops} Theorem 6.7). 
\vspace{5pt}\\
\indent In this  section of this thesis, I attempted to highlight the important theorems and results of mathematicians who have delve into the subject of coverings of groups and coverings of other algebraic structures since the time of G.A. Miller near the beginning of the last century. A lot has been accomplished that a whole 20+ page thesis would be needed to cover more general results of the papers mentioned in this section and more. In the following section, we attempt derive some theorems of groups that have equal coverings. One thing to note that we may need to keep our eyes peeled for groups and loops of prime exponent since there have been at least two separate instances where such groups seem to correlate with being the union of equal order proper subgroups.
\section{Preliminaries for Equal Coverings}
Recall that if $G$ is a group, then an equal covering of $G$ is a collection of proper subgroups such that their union is $G$ and all such subgroups are of the same order. Again, since all cyclic groups already do not have a covering, we will focus on non-cyclic groups for the remainder of this paper. So, unless otherwise specified, in future theorems we will restrict ourselves to finite non-cyclic groups. The first theorem of this section will be powerful, but first we must mention the concept of the exponent of a group.
\begin{definition}
If $G$ is a group, then the \textbf{exponent} of $G$ is the smallest positive integer $n$ for which $a^n = 1$. We will use $\exp(G)$ to denote the exponent of $G$.
\end{definition}
\begin{remark}
If $G$ is a finite group, then the exponent of $G$ is the least common multiple of all the orders of the elements of $G$. 
\end{remark}
\begin{theorem}\label{ExpTheorem}
If $G$ has an equal covering $\Pi = \{H_i\}$, then $\exp(G)$ divides $|H_i|$ for all $H_i \in \Pi$.
\end{theorem}
\begin{proof}
Let $\Pi = \{H_i\}$ be an equal covering of $G$ and suppose $x \in G$. Since $\Pi$ is a covering, $x \in H$ for some $H \in \Pi$. Since $|x|$ divides $|H|$, $|x|$ divides the order of $H_i$ for all $H_i \in \Pi$, since $\Pi$ is an equal covering. It follows then the order of every element of $G$ divides the order of every $H_i \in \Pi$, so $\exp(G)$ divides $|H_i|$ for all $H_i \in \Pi$.
\end{proof}
\begin{corollary}\label{ExpCor}
If $\exp(G) \nmid |K|$ for every maximal subgroup $K$ of $G$, then $G$ does not have an equal covering.
\end{corollary}
Now, recall $D_{2n}$ is our notation for the dihedral group of order $2n$. That is, let $D_{2n} = \langle r,s \rangle$, where the defining equations are $r^n = s^2 = 1$ and $srs = r^{-1}$. It turns out that there is a way to determine whether a dihedral group has an equal covering - and even more, we simply must examine the parity of $n$. As we will see, $D_{2n}$ will have an equal covering if and only if $n$ is even. 
\begin{lemma}\label{OrderDn}
In $D_{2n}$, if $i \in \{1,2,...,n\}$ then $|r^is| = |sr^i| = 2$ and $|r^i| = \lcm(n,i)/i$.
\end{lemma}
\begin{proof}
Using the fact that $srs = r^{-1}$, we must have $(srs)^i = sr^is = r^{-i}$ using induction. Now, multiplying $r^i$ on both sides of $sr^is = r^{-i}$ will result in $(r^is)(r^is) = (sr^i)(sr^i) = 1$.\vspace{5pt}\\
We have $(r^i)^{\lcm(i,n)/i} = r^{\lcm(i,n)} = 1$, since $\lcm(i,n)$ is divisible by $n$, the order of $r$. 
\end{proof}
\begin{corollary}\label{ExpDn}
If $n$ is odd then $\exp(D_{2n}) = 2n$, if $n$ is even then $\exp(D_{2n}) = n$. In other words, $\exp(D_{2n}) = \lcm(n,2)$.
\end{corollary}
\begin{proof}
By Lemma \ref{OrderDn}, we must have that $\exp(G)$ must be divisible by 2 and must divide $\lcm(i,n)$ for all $i \in \{1,2,...,n\}$. Observe when $i$ and $n$ are coprime, then $\lcm(i,n) = i\cdot n$, and so $|\langle r^i \rangle| = i\cdot n/i = n$. This suggests $\exp(D_{2n})$ must be divisible by $n$. If $n$ is odd, then the only possible value for $\exp(D_{2n})$ must be $2n$ since it will be the smallest multiple of $n$ and $2$ that also divides the order of the group. If $n$ is even, then $\exp(D_{2n}) = n$ since $n$ will be divisible by 2 and it is the largest proper divisor of $2n$. Therefore, $\exp(D_{2n}) = \lcm(n,2)$.
\end{proof}

\begin{theorem}\label{EqCovDn}
(i) 
If $n$ is odd, $D_{2n}$ has no equal covering. (ii) If $n$ is even, then $D_{2n}$ has an equal covering $\Pi = \{\langle r \rangle, \langle r^2, s\rangle, \langle r^2, rs\rangle\}$. Consequently, $\sigma(D_{2n}) = 3$ for even $n$.
\end{theorem}
\begin{proof}
(i) Let $n$ be odd and suppose $D_{2n}$ had an equal covering. Then there would be some maximal subgroup whose order is divisible by $2n$, by Corollary \ref{ExpCor} and Corollary \ref{ExpDn}. Since $2n$ is the order of $D_{2n}$, we reach a contradiction.\vspace{5pt}\\
% Revisit this part of the proof
(ii) Let $n$ be even, $A = \langle r^2, s\rangle$ and $B = \langle r^2, rs\rangle$. We will first prove any $d \in D_{2n}$ lies in at least one of $\langle r \rangle$, $A$, or $B$. Then, we will show $|\langle r \rangle| = |A| =|B|$.\vspace{5pt}\\
For any $d \in D_{2n}$, we have the following three cases: $d$ is of the form $r^i$, $r^is$ when $i$ is even, or $r^is$ when $i$ is odd.\\
If $d = r^i$, then $d \in \langle r \rangle$.\\
If $d = r^is$, where $i$ is even, then $d = r^{2k}s$ for some $k$. Since $r^{2k}s = (r^2)^ks \in A$, $d \in A$.\\
If $d = r^is$, where $i$ is odd, then $d = r^{2k+1}s$, for some $k$. Since $r^{2k+1}s = (r^2)^k(rs) \in B$, $d \in B$. So, $\Pi$ is at least a covering of $D_{2n}$, which implies $\sigma(D_{2n}) = 3$.\vspace{5pt}\\
We know $|r| = n$, so we will now show $|A| = |B| = n$.\\
First, any element of $A$ is either an even power of $r$, or an even power of $r$ multiplied by $s$. Since $n$ is even, any power of $r^2$ will be even, and if we multiply any such elements with $s$, we simply obtain an even power of $r$ multiplied with $s$. Since $n$ is even, and $n$ is the order of $r$, the number of even powers of $r$ is $\frac{n}{2}$. Since we multiply such numbers by $s$, we obtain $\frac{n}{2}$ new elements. It follows $|A| = \frac{n}{2} + \frac{n}{2} = n$.\\
Now, any element of $B$ is either an even power of $r$, or an even power of $r$ multiplied by $rs$. We know the number of even powers of $r$ is $\frac{n}{2}$. Multiplying such numbers by $rs$, we obtain elements of the form $(r^{2k})(rs) = r^{2k+1}s$, which are odd powers of $r$ multiplied by $s$, so we obtain $\frac{n}{2}$ new elements. It follows $|B| = \frac{n}{2} + \frac{n}{2}= n$.\vspace{5pt}\\
Therefore, $\Pi = \{\langle r \rangle, \langle r^2, s\rangle, \langle r^2, rs\rangle\}$ is an equal covering of $D_{2n}$ when $n$ is even.
\end{proof}
While we are in search finding significant theorems about particular family of groups such as Theorem \ref{EqCovDn}, we have developed some other useful theorems. To start, as mentioned prior, Isaacs had proven that non-cyclic $p$-groups of exponent $p$ had an equal partition. Since equal partitions are a particular type of equal coverings, we will demonstrate that all non-cyclic $p$-groups, no matter their exponent, will have an equal covering. From here, we will develop some more useful theorems that can aid us in determining whether a group has an equal covering.
\begin{theorem}\label{pgroups}
Every non-cyclic finite $p$-group has an equal covering.
\end{theorem}
\begin{proof}
Let $G$ be a non-cyclic group of order $p^n$, where $p$ is prime and $n > 1$. If $x \in G$, then $x$ must lie in  some maximal subgroup of $G$, say $M_x$. Since $G$ is a $p$-group, and thus nilpotent,  $[G:M_x]$ is prime and consequently, $[G:M_x] = p$ by J.J. Rotman  (\cite{rotman2012introduction}, Theorem 4.6 (ii)). It follows for every $x \in G$, $x$ lies in a maximal subgroup $M_x$ such that $|M_x| = p^{n-1}$. It follows the collection of all such subgroups forms an equal covering of $G$. 
\end{proof}
\begin{theorem}\label{direct2}
If $G$ or $H$ have an equal covering, then $G \times H$ has an equal covering.
\end{theorem}
\begin{proof}
Without loss of generality, suppose $X = \{G_i: i \in I\}$ is an equal covering for $G$, and consider the set $\Pi = \{G_i \times H: i \in I\}$.\vspace{5pt}\\
Suppose $g \in G$ and $h \in H$. Since $X$ is a covering for $G$, $(g,h) \in G_i \times H$ for some $i \in I$. Since $X$ is an equal covering for $G$, if $i,j \in I$, then $|G_i| = |G_j|$, and consequently $|G_i \times H| = |G_i||H| = |G_j||H| = |G_j \times H|$. It follows every subgroup of $G \times H$ in $\Pi$ is of the same  order. Therefore, $\Pi$ is an equal covering of $G \times H$.
\end{proof}
\begin{corollary}\label{directarbitrary}
If $G = \prod\limits_{i}H_i$ is a direct product of groups, and at least one $H_i$ has an equal covering, then $G$ has an equal covering.
\end{corollary}
\noindent Corollary \ref{directarbitrary} is actually a powerful theorem since if we happen to know a certain group $G_1$ has an equal covering, then we easily construct a new group $G_2$ that has an equal covering by taking any direct product of finite groups with $G_1$, not needing to know anything about the remaining groups. 
\begin{theorem}\label{nilpotent}
If $G$ is a finite non-cyclic nilpotent group then it must have an equal covering.
\end{theorem}
\begin{proof}
Let $G$ be a finite non-cyclic nilpotent group. Since $G$ is a direct of its Sylow $p$-subgroups, by (\cite{rotman2012introduction}, Theorem 5.39), and all $p$-groups have an equal covering, so does $G$ by Theorem \ref{pgroups} and Corollary \ref{directarbitrary}. 
\end{proof}
\begin{theorem}\label{distinctp}
Suppose $G$ is a group whose order is the product of two or more distinct primes, each with a multiplicity of one. That is, $|G| = p_1p_2...p_n$, where $p_i =p_j$ only when $i=j$. Then $G$ cannot have an equal covering.
\end{theorem}
\begin{proof}
Suppose $p_i$ is some prime divisor of $|G|$. It follows $G$ must have some element $x_i \in G$ such that $|x_i| = p_i$. Since the order every element of $G$ is either a prime $p_i$ or a product of prime divisors of $|G|$, we must have $\exp(G) = \text{lcm}\{|x|: x \in G\} = p_1 \cdot p_2 \cdot ... \cdot p_n = |G|$. By Corollary \ref{ExpCor}, since there does not exist a maximal subgroup $K$ such that $|G| \mid |K|$, $G$ does not have an equal covering.
\end{proof}
\begin{theorem}\label{Quotient}
If $H \vartriangleleft G$ and $G/H$ has an equal covering, then $G$ has an equal covering.
\end{theorem}
\begin{proof}
Suppose $G/H$ has the equal covering $\Pi = \{N_1, N_2, ..., N_k\}$. Since $N_i < G/H$ and $H \vartriangleleft G$, $N_i = N_i^*/H$ for some $N_i^* < G$ by Correspondence Theorem in (\cite{rotman2012introduction}, Theorem 2.28). It follows that if $g \in G$ and $gH \in N_i$, then $g \in N_i^*$. So, the set $\Gamma = \{N_1^*, N_2^*, ..., N_k^*\}$ is a covering of $G$. Finally, for any $N_i,N_j \in P_i$, we have $|N_i| = |N_j|$, implying $|N_i^*/H| = |N_j^*/H|$, or $|N_i^*| = |N_j^*|$. Thus, $\Gamma$ is an equal covering of $G$.
\end{proof}
    Theorem \ref{Quotient} can provide us an alternate proof to Theorem \ref{direct2}. If $G = H \times K$, with $H $ having an equal covering, then we have $G/K \cong H$ has an equal covering. By Theorem \ref{Quotient}, so does $G$. The advantage of the original proof for Theorem \ref{direct2} is that if we know what an equal covering of  $G$ is then we can then easily find an example of an equal covering for $G \times H$.
 In a moment, we will define the semidirect product of two groups, as it turns out a lot of groups can be written in terms of the semidirect product of two groups. 
\begin{definition}
Let $G$ be a group with subgroups $H$ and $K$. We say $G$ is the semidirect product of $H$ and $K$, if $H \vartriangleleft G$ and $Q \cong K$ is a complement of $H$, that is, $HQ = G$ and $H \cap Q = \{1\}$. If $G$ is the semidirect product of $H$ and $K$, we denote this by $G = H \rtimes K$.
\end{definition}
\begin{remark}\label{QuotientSemi}
Note that if $G = H \rtimes K$, then $K \cong G/H$.
\end{remark}

\begin{corollary}\label{covering-semi}
If $G = H \rtimes K$ and $K$ has an equal covering, then $G$ has an equal covering.
\end{corollary}
\begin{proof}
By Remark \ref{QuotientSemi}, $K \cong G/H$. By Theorem \ref{Quotient}, if $K$ has an equal covering then so does $G$.
\end{proof}
Let us recall that if $G$ is a group with a subgroup $H$ of index 2, then $H \vartriangleleft G$. If $|G| > 2$, then if $G$ possesses such a group then it cannot be simple. What follows is a proposition.
\begin{proposition}\label{simpleexp}
If $G$ is a finite non-cyclic simple group with $\exp(G) = |G|/2$, then $G$ does not have an equal covering.
\end{proposition}
\begin{proof}
If $G$ is a group with an equal covering $\Pi$ and $\exp(G) = |G|/2$, then it follows that $G$ has a maximal subgroup $M$ with whose order is $|G|/2$. It follows $[G:M] = 2$ and thus $M \vartriangleleft G$. This contradicts the simplicity of $G$.
\end{proof}
\begin{remark}
Although $S_4$ is not simple, since $A_4$ is a normal subgroup, it does not have an equal covering. This is due to the exponent of $S_4$ being $12$ and $A_4$ is the only subgroup of order 12.
\end{remark}
We have here in this section a collection of useful theorems which will help deduce if a group hasn't equal covering or not
\section{Possible Future Work}
I took interest in this topic due to Dr. Tuval Foguel bringing it to my attention, and after having done the research in literature and receiving his help in deriving the theorems in Section 3, I hope to continue researching into this topic.  The following are questions or ideas I would like to go delve into more once I become familiar with more algebra concepts and definitions, as well as the literature related to this topic:
\begin{enumerate}
    \item Develop a more sophisticated \texttt{\texttt{GAP}} code that can be faster and efficient at determining which finite groups have an equal covering and which ones do not.
    \item For those that do have an equal covering, what would the equal covering number $\varepsilon(G)$ be, that is, 
    \begin{equation*}
        \varepsilon(G) = \min\{|\Pi|\ : \Pi\ \text{is an equal covering of}\ G\}
    \end{equation*}
    \item In addition to developing more sophisticated code, I would like to derive more theorems that indicate which finite groups do (or do not) have an equal coverings. Possibly develop a classification system of which groups have an equal covering, similar to how there is a classification system of partitionable groups.
    \item If a classification system somehow not feasible, I still would like to examine certain type of groups such as the finite simple groups one would find in the ATLAS \cite{conway1985finite}. I would love to either prove or determine the first counterexample to the conjecture at the end of Section 4.
\end{enumerate}
\section{Appendix}
\subsection*{Equal Covering of Statuses of Groups}
This section is dedicated, we will attempt to state which groups up to order 60 as well all finite non-cyclic groups up to the Mathieu group $M_{12}$ have equal coverings (or don't) based on the theorems presented in the Section 3 as well as use \texttt{GAP} \cite{GAP4}.\texttt{GAP} is a software that can be utilized to examine established groups or derive new ones from a computational standpoint, such as determining the order of a group or finding the conjugacy classes of groups.\\
We hope to examine groups that are at least of order 60 and below, as well as the first couple of groups in the ATLAS \cite{conway1985finite}. To start off, we will list out all order of groups up to order 60 that we know groups of such orders will have or not have equal coverings as well as state the reason for the status, where T\# will mean the same as Theorem \#.\vspace{5pt}
\begin{center}
\renewcommand{\arraystretch}{2}
   \begin{tabular}{|c|c|c|}\hline
    Order & Have An Equal Covering? & Reason \\\hline
    1 & No & Trivial\\\hline
    All primes $\leq 60$ & No & T\ref{Cyclic}\\\hline
    4, 8, 9, 16, 25, 27, 32, 49 & Yes (for non-cyclic groups) & T\ref{pgroups}\\ \hline
    \makecell{6, 10, 14, 15, 21, 22, 26, 30, 33, 34,\\ 35, 38, 39, 42, 46, 51, 55, 57, 58} & No & T\ref{distinctp}\\\hline
\end{tabular}\vspace{5pt}\\
Table 1: Groups of Order 60 or Less with Equal Covering Status Based on Order
\end{center}\vspace{5pt}
As one will notice, we are missing the following numbers from the list of the first 60 natural numbers: 12, 18, 20, 24, 28, 36, 40, 44, 45, 48, 50, 52, 54, 56, 60. For each group of each of these orders, we will determine whether they have an equal covering by theorems provided as well as using \texttt{GAP} and ATLAS\cite{conway1985finite} for numerical arguments. The ATLAS is an encyclopedia of finite simple groups such as $A_{11}$ or the Mathieu group $M_{12}$, providing the different families of such groups and for each finite simple group characteristics such as the classification of their maximal subgroups. For each order, the sequence in which we will examine the groups is by the Group ID that is based in \texttt{GAP}'s identification system. For example, the ID of $S_3$ in \texttt{GAP} is [6,1] indicating it is listed as the first group of order 6. Before we go into this part of the section, let us note two things:
\begin{itemize}
    \item that we omit $\mathbb{Z}_n$ for each $n$ since they trivially have no covering;
    \item to identify a group's structure we will use \texttt{GAP} and other sources. Although \texttt{GAP}'s\\ \texttt{StructureDescription} function is useful for this, there will be instances where groups with different ID's will be called the same name. Example: In \texttt{GAP}, groups [20,1] and [20,3] are called "\texttt{C5 : C4}" (or $\mathbb{Z}_5 \rtimes \mathbb{Z}_4$). In which case we will write (1) after the group name to indicate it is the first group with that name, (2) if it is the second group with that name, and so on. \textbf{In some instances, we may write the group in a way where it will be clear if a theorem or other statement was used to determine the equal covering status.}
\end{itemize}
\noindent \textbf{How to Read Table 2:}
\begin{enumerate}
    \item The first column is \textbf{Group ID}, which is the I.D. number of the corresponding group. We read $[m,n]$ to mean this the $n$th group of order $m$.
    \item The second column is the \textbf{name(s)} of the group.
    \item The third column is the \textbf{Exponent} of the group.
    \item The fourth column indicate whether the given group is \textbf{nilpotent} or not. 
    \item The fifth column is the status of whether the group has an \textbf{equal covering or not}. 
    \item The sixth column is the reason why the group has or does not have an equal covering. We will write either write the Corollary(ies),Theorem(s), etc. used and that were mentioned in Section 3and write either C\# (for Corollary \#), T\# (for Theorem \#), R\# (for Remark \#), P\# (for Proposition \#), or state \texttt{GAP} if primarily \texttt{\texttt{GAP}} code was used to determine the status.
\end{enumerate}
\begin{center}
\begin{longtable}{| p{.1\textwidth} | p{.2\textwidth} |p{.1\textwidth}| p{.1\textwidth}|p{.2\textwidth}|p{.1\textwidth}|p{.3\textwidth}|} \hline
Group ID & Name(s) & Exponent & Nilpotent? & Equally Coverable? & Reason\\ \hline 
[12, 1] & $\mathbb{Z}_3 \rtimes \mathbb{Z}_4 \cong Q_{12}$ & 12 & No & No & C\ref{ExpCor} \\ \hline
[12, 3] & $A_4$ & 6 &  No & No & \texttt{GAP} \\ \hline
[12, 4] & $D_{12}$ & 6 & No & Yes & T\ref{EqCovDn}\\ \hline
[12, 5] & $\mathbb{Z}_2^2 \times \mathbb{Z}_3$ &  6 & Yes & Yes & C\ref{directarbitrary}\\ \hline
Group ID & Name(s) & Exponent & Nilpotent? & Equally Coverable? & Reason\\ \hline 
[18, 1] & $D_{18}$ & 18 & No & No & C\ref{directarbitrary} \\ \hline
[18, 3] & $S_3 \times \mathbb{Z}_3$ & 6 & No & No & \texttt{GAP}\\ \hline
[18, 4] & $\mathbb{Z}_3^2 \rtimes \mathbb{Z}_2$ & 6 & No & Yes & \texttt{GAP} \\ \hline
[18, 5] & $\mathbb{Z}_3^2 \times \mathbb{Z}_2$ & 6 & Yes & Yes & T\ref{EqCovDn} \\ \hline
[20, 1] & $Q_{20}$ & 20 & No & No & C\ref{ExpCor}\\ \hline
[20, 3] & $\mathbb{Z}_5 \rtimes \mathbb{Z}_4$ & 20 & No & No & C\ref{ExpCor} \\ \hline
[20, 4] & $D_{20}$ & 10 & No &  Yes & C\ref{directarbitrary} \\ \hline
[20, 5] & $\mathbb{Z}_2^2 \times \mathbb{Z}_5$ & 10 & Yes & Yes & T\ref{EqCovDn}\\ \hline
[24, 1] & $\mathbb{Z}_3 \rtimes \mathbb{Z}_8$ & 24 & No & No & C\ref{ExpCor}\\ \hline
[24, 3] & SL(2,3) $ \cong Q_8 \rtimes \mathbb{Z}_3$ & 12 & No & No & \texttt{GAP} \\ \hline
[24, 4] & $\mathbb{Z}_3 \rtimes Q_8 \cong Q_{24}$ & 12 & No & Yes & C\ref{covering-semi}\\ \hline
[24, 5] & $S_3 \times \mathbb{Z}_4$ & 12 & No & Yes & \texttt{GAP} \\ \hline
[24, 6] & $D_{24}$ & 12 & No & Yes & C\ref{directarbitrary} \\ \hline
[24, 7] & $Q_{12} \times \mathbb{Z}_2$ & 12 & No & Yes & \texttt{GAP}\\ \hline
[24, 8] & $(\mathbb{Z}_3 \times \mathbb{Z}_2^2) \rtimes \mathbb{Z}_2$ & 12 & No & Yes & \texttt{GAP}\\ \hline
[24, 9] & $\mathbb{Z}_{12} \times \mathbb{Z}_2$ & 12 & Yes & Yes & T\ref{nilpotent}\\ \hline
[24, 10] & $D_8 \times \mathbb{Z}_3$ & 12 & Yes & Yes & C\ref{directarbitrary}\\ \hline
[24, 11] & $Q_8 \times \mathbb{Z}_3$ & 12 & Yes & Yes & T\ref{nilpotent} \\ \hline
[24, 12] & $S_4$ & 12 & No & No & R3\\\hline
[24, 13] & $A_4 \times \mathbb{Z}_2$ & 6 & No & No & \texttt{GAP}\\ \hline
[24, 14] & $S_3 \times \mathbb{Z}_2^2$ & 6 & No  & Yes & C\ref{directarbitrary}\\ \hline
[24, 15] & $\mathbb{Z}_2^3 \times \mathbb{Z}_3$ & 6 & Yes & Yes & C\ref{directarbitrary}\\ \hline
[28, 1] & $\mathbb{Z}_7 \rtimes \mathbb{Z}_4$ & 28 & No & No & C\ref{ExpCor}\\ \hline
[28, 3] & $D_{28}$ & 14 & No & Yes & T\ref{EqCovDn} \\ \hline
[28, 4] & $\mathbb{Z}_2^2 \times \mathbb{Z}_7$ & 14 & Yes & Yes & C\ref{directarbitrary}\\ \hline
[36, 1] & $\mathbb{Z}_9 \rtimes \mathbb{Z}_4$ & 36 & No & No & C\ref{ExpCor}\\ \hline
[36, 3] & $\mathbb{Z}_2^2 \rtimes \mathbb{Z}_9$ & 18 & No & No & \texttt{GAP} \\ \hline
[36, 4] & $D_{36}$ & 18 & No & Yes & T\ref{EqCovDn} \\ \hline
[36, 5] & $\mathbb{Z}_2^2 \times \mathbb{Z}_9$ & 18 & Yes & Yes & T\ref{nilpotent}\\\hline
[36, 6] & $(\mathbb{Z}_3 \rtimes \mathbb{Z}_4) \times \mathbb{Z}_3$ & 12 & No & No & \texttt{GAP} \\ \hline
[36, 7] & $\mathbb{Z}_3^2 \rtimes \mathbb{Z}_4$(1) & 12 & Yes & Yes & \texttt{GAP} \\\hline
[36, 8] & $\mathbb{Z}_3^2 \times \mathbb{Z}_4$ & 12 & Yes & Yes & C\ref{directarbitrary} \\ \hline
[36, 9] & $\mathbb{Z}_3^2 \rtimes \mathbb{Z}_4$(2) & 12 & No & No & \texttt{GAP} \\\hline
[36, 10] & $S_3 \times S_3$ & 6 & No & Yes & \texttt{GAP}\\ \hline
[36, 11] & $A_4 \times \mathbb{Z}_3$ & 6 & No & Yes & \texttt{GAP}\\ \hline
[36, 12] & $S_3 \times \mathbb{Z}_6$ & 6 & No & Yes & \texttt{GAP}\\ \hline 
Group ID & Name(s) & Exponent & Nilpotent? & Equally Coverable? & Reason\\ \hline 
[36, 13] & $(\mathbb{Z}_3^2 \rtimes \mathbb{Z}_2) \times \mathbb{Z}_2$ & 6 & No & Yes & \texttt{GAP}\\ \hline
[36, 14] & $\mathbb{Z}_3^2 \times \mathbb{Z}_2^2$ & 6 & Yes & Yes & C\ref{directarbitrary}\\ \hline
[40, 1] & $\mathbb{Z}_5 \rtimes \mathbb{Z}_8$(1) & 40 & No & No & C\ref{ExpCor}\\ \hline
[40, 3] & $\mathbb{Z}_5 \rtimes \mathbb{Z}_8$(2) & 40 & No & No & C\ref{ExpCor}\\ \hline
[40, 4] & $\mathbb{Z}_5 \rtimes Q_8$ & 20 & No & Yes & C\ref{covering-semi}\\ \hline
[40, 5] & $D_{10} \times \mathbb{Z}_4 $ & 20 & No & Yes & \texttt{GAP}\\ \hline
[40, 6] & $D_{40}$ & 20 & No & Yes & T\ref{EqCovDn}\\ \hline
[40, 7] & $(\mathbb{Z}_5 \rtimes \mathbb{Z}_4) \times \mathbb{Z}_2$ & 20 & No & Yes & \texttt{GAP}\\ \hline
[40, 8] & $(\mathbb{Z}_{2}^2 \times \mathbb{Z}_5) \rtimes \mathbb{Z}_2$ & 20 & No & Yes & \texttt{GAP}\\ \hline
[40, 9] & $\mathbb{Z}_5 \times \mathbb{Z}_4 \times \mathbb{Z}_2$ & 20 & Yes & Yes & C\ref{directarbitrary}\\ \hline
[40, 10] & $D_8 \times \mathbb{Z}_5$ & 20 & Yes & Yes & C\ref{directarbitrary}\\\hline
[40, 11] & $Q_8 \times \mathbb{Z}_5$ & 20 & Yes & Yes & T\ref{nilpotent}\\\hline
[40, 12] & $(\mathbb{Z}_5 \rtimes \mathbb{Z}_4) \times \mathbb{Z}_2$ & 20 & No  & Yes & \texttt{GAP}\\\hline
[40, 13] & $D_{10} \times \mathbb{Z}_2^2$ & 10 & No & Yes & C\ref{directarbitrary}\\\hline
[40, 14] & $\mathbb{Z}_2^3\times \mathbb{Z}_5$ & 10 & Yes & Yes & C\ref{directarbitrary}\\\hline
[44, 1] & $\mathbb{Z}_{11} \rtimes \mathbb{Z}_4$ & 44 & No & No & C\ref{ExpCor}\\\hline
[44, 3] & $D_{44}$ & 22 & No & Yes & T\ref{EqCovDn}\\\hline
[44, 4] & $\mathbb{Z}_2^2 \times \mathbb{Z}_{11}$ & 22 & Yes & Yes & C\ref{directarbitrary}\\\hline
[45, 2] & $\mathbb{Z}_3^2 \times \mathbb{Z}_5$ & 15 & Yes & Yes & C\ref{directarbitrary}\\\hline
[48, 1] & $\mathbb{Z}_3 \rtimes \mathbb{Z}_{16}$ & 48 & No  & No & C\ref{ExpCor}\\\hline
[48, 3] & $\mathbb{Z}_4^2 \rtimes \mathbb{Z}_3$ & 12 & No & Yes & \texttt{GAP}\\\hline
[48, 4] & $S_3 \times \mathbb{Z}_8$ & 24 & No & Yes & \texttt{GAP}\\\hline
[48, 5] & $\mathbb{Z}_{24} \rtimes \mathbb{Z}_2$(1) & 24 & No & Yes & \texttt{GAP}\\\hline
[48, 6] & $\mathbb{Z}_{24} \rtimes \mathbb{Z}_2$(2) & 24 & No & Yes & \texttt{GAP}\\\hline
[48, 7] & $D_{48}$ & 24 & No  & Yes & T\ref{EqCovDn}\\\hline
[48, 8] & $\mathbb{Z}_3 \rtimes Q_{16}$(1) & 24 & No & Yes & C\ref{covering-semi}\\\hline
[48, 9] & $(\mathbb{Z}_3 \rtimes \mathbb{Z}_8) \times \mathbb{Z}_2$ & 24 & No & Yes & \texttt{GAP}\\\hline
[48, 10] & $(\mathbb{Z}_3 \rtimes \mathbb{Z}_8) \rtimes \mathbb{Z}_2$(1) & 24 & No & Yes & \texttt{GAP}\\ \hline
[48, 11] & $(\mathbb{Z}_3 \rtimes \mathbb{Z}_4) \times \mathbb{Z}_4$ & 12 & No  & Yes & \texttt{GAP}\\\hline
[48, 12] & $(\mathbb{Z}_3 \rtimes \mathbb{Z}_4) \rtimes \mathbb{Z}_4$ & 12 & No  & Yes & \texttt{GAP}\\\hline
[48, 13] & $\mathbb{Z}_{12} \rtimes \mathbb{Z}_4$ & 12 & No & Yes & \texttt{GAP}\\\hline
[48, 14] & $(\mathbb{Z}_{12} \times \mathbb{Z}_2) \rtimes \mathbb{Z}_2$(1) & 12 & No & Yes & \texttt{GAP}\\\hline
[48, 15] & $(D_8 \times \mathbb{Z}_3) \rtimes \mathbb{Z}_2$ & 24 & No & Yes & \texttt{GAP}\\\hline
[48, 16] & $(\mathbb{Z}_3 \rtimes \mathbb{Z}_8) \rtimes \mathbb{Z}_2$(2) & 24 & No & Yes & \texttt{GAP}\\\hline
[48, 17] & $(Q_8 \times \mathbb{Z}_3) \rtimes \mathbb{Z}_2$ & 24 & No & Yes & \texttt{GAP}\\\hline
[48, 18] & $\mathbb{Z}_3 \rtimes Q_{16}$(2) & 24 & No & Yes & C\ref{covering-semi}\\\hline
Group ID & Name(s) & Exponent & Nilpotent? & Equally Coverable? & Reason\\ \hline 
[48, 19] & $((\mathbb{Z}_3 \rtimes \mathbb{Z}_4) \times \mathbb{Z}_2) \rtimes \mathbb{Z}_2$ & 12 & No & Yes & \texttt{GAP}\\\hline
[48, 20] & $\mathbb{Z}_4^2 \times \mathbb{Z}_3$ & 12 & Yes & Yes & C\ref{directarbitrary}\\\hline
[48, 21] & $((\mathbb{Z}_4 \times \mathbb{Z}_2) \rtimes \mathbb{Z}_2) \times \mathbb{Z}_3$ & 12 & Yes & Yes & T\ref{nilpotent}\\\hline
[48, 22] & $(\mathbb{Z}_4 \rtimes \mathbb{Z}_4) \times \mathbb{Z}_3$ & 12 & Yes & Yes & T\ref{nilpotent}\\\hline
[48, 23] & $\mathbb{Z}_{8} \times \mathbb{Z}_3 \times \mathbb{Z}_2$ & 24 & Yes & Yes & T\ref{nilpotent}\\\hline
[48, 24] & $(\mathbb{Z}_8 \rtimes \mathbb{Z}_2) \times \mathbb{Z}_3$ & 24 & Yes & Yes & T\ref{nilpotent}\\\hline
[48, 25] & $D_{16} \times \mathbb{Z}_3$ & 24 & Yes & Yes & C\ref{directarbitrary}\\\hline
[48, 26] & $QD_{16} \times \mathbb{Z}_3$ & 24 & Yes & Yes & C\ref{directarbitrary}\\\hline
[48, 27] & $Q_{16} \times \mathbb{Z}_3$ & 24 & Yes & Yes & C\ref{directarbitrary}\\\hline
[48, 28] & $\langle 2,3,4\rangle$ & 24 & No & No & \texttt{GAP} \\\hline
[48, 29] & GL$(2,3)$ & 24 & No & No & \texttt{GAP}\\\hline
[48, 30] & $A_4 \rtimes \mathbb{Z}_4$ & 12 & No & No & \texttt{GAP}\\\hline
[48, 31] &$A_4 \times \mathbb{Z}_4$ & 12 & No & No & \texttt{GAP}\\\hline
[48, 32] & SL$(2,3) \times \mathbb{Z}_2$ & 12 & No & No & \texttt{GAP}\\\hline
[48, 33] & SL$(2,3) \rtimes \mathbb{Z}_2$ & 12 & No & No & \texttt{GAP} \\\hline
[48, 34] & $(\mathbb{Z}_3 \rtimes Q_8) \times \mathbb{Z}_2$  & 12 & No & Yes & C\ref{covering-semi}\\\hline
[48, 35] & $S_3 \times \mathbb{Z}_4 \times \mathbb{Z}_2$ & 12 & No & Yes & \texttt{GAP}\\\hline
[48, 36] & $D_{24} \times \mathbb{Z}_2$ & 12 & No & Yes & C\ref{directarbitrary}\\\hline
[48, 37] & $(\mathbb{Z}_{12} \times \mathbb{Z}_2) \rtimes \mathbb{Z}_2$(2) & 12 & No & Yes & \texttt{GAP}\\\hline
[48, 38] & $D_8 \times S_3$ & 12 & No & Yes & C\ref{directarbitrary}\\\hline
[48, 39] & $((\mathbb{Z}_3 \rtimes \mathbb{Z}_4) \times \mathbb{Z}_2) \rtimes \mathbb{Z}_2$ & 12 & No & Yes & \texttt{GAP}\\\hline
[48, 40] & $Q_8 \times S_3$ & 12 & No & Yes & C\ref{directarbitrary}\\\hline
[48, 41] & $(S_3  \times \mathbb{Z}_4) \rtimes \mathbb{Z}_2$ & 12 & No & Yes & \texttt{GAP}\\\hline
[48, 42] & $(\mathbb{Z}_3 \rtimes \mathbb{Z}_4) \times \mathbb{Z}_2^2$ & 12 & No & Yes & C\ref{directarbitrary}\\\hline
[48, 43] & $((\mathbb{Z}_6 \times \mathbb{Z}_2) \rtimes \mathbb{Z}_2) \times \mathbb{Z}_2$ & 12  & No & Yes & \texttt{GAP}\\\hline
[48, 44] & $\mathbb{Z}_2^2 \times \mathbb{Z}_{12}$ & 12 & Yes & Yes & C\ref{directarbitrary}\\\hline
[48, 45] & $D_8 \times \mathbb{Z}_6$ & 12 & Yes & Yes & C\ref{directarbitrary}\\\hline
[48, 46] & $Q_8 \times \mathbb{Z}_6$ & 12 & Yes & Yes & C\ref{directarbitrary}\\\hline
[48, 47] & $((\mathbb{Z}_4  \times \mathbb{Z}_2) \rtimes \mathbb{Z}_2) \times \mathbb{Z}_3$ & 12 & Yes & Yes & T\ref{nilpotent}\\\hline
[48, 48] & $S_4 \times \mathbb{Z}_2$ & 12 & No & Yes & \texttt{GAP}\\\hline
[48, 49] & $A_4 \times \mathbb{Z}_2^2$ & 6 & No & Yes & C\ref{directarbitrary}\\\hline
[48, 50] & $\mathbb{Z}_2^4 \rtimes \mathbb{Z}_3$ & 6 & No & Yes & \texttt{GAP}\\\hline
[48, 51] & $\mathbb{Z}_2^3 \times S_3$  & 6 & No & Yes & C\ref{directarbitrary}\\\hline
[48, 52] & $\mathbb{Z}_2^3 \times \mathbb{Z}_3$ & 6 & Yes & Yes & C\ref{directarbitrary}\\\hline
[50, 1] & $D_{50}$ & 50 & No & No & C\ref{ExpCor}\\\hline
[50, 3] & $D_{10} \times \mathbb{Z}_5$ & 10 & No & No & \texttt{GAP}\\\hline
Group ID & Name(s) & Exponent & Nilpotent? & Equally Coverable? & Reason\\ \hline 
[50, 4] & $\mathbb{Z}_5^2 \rtimes \mathbb{Z}_2$ & 10 & No & Yes & \texttt{GAP}\\\hline
[50, 5] & $ \mathbb{Z}_5^2 \times \mathbb{Z}_2$ & 10 & Yes & Yes & C\ref{directarbitrary}\\\hline
[52, 1] & $\mathbb{Z}_{13} \rtimes \mathbb{Z}_4$(1) & 52 & No & No & C\ref{ExpCor}\\\hline
[52, 3] & $\mathbb{Z}_{13} \rtimes \mathbb{Z}_4$(2) & 52 & No & No & C\ref{ExpCor}\\\hline
[52, 4] & $D_{52}$ & 26 & No & Yes & T\ref{EqCovDn}\\\hline
[52, 5] & $\mathbb{Z}_2^2 \times \mathbb{Z}_{13}$ & 26 & Yes & Yes & C\ref{directarbitrary}\\\hline
[54, 1] & $D_{54}$ & 54 & No & No & T\ref{EqCovDn}\\\hline
[54, 3] & $D_{18} \times \mathbb{Z}_3$ & 18 & No & No & \texttt{GAP}\\\hline
[54, 4] & $S_3 \times \mathbb{Z}_9$ & 18 & No & No & \texttt{GAP}\\\hline
[54, 5] & $(\mathbb{Z}_3^2 \rtimes \mathbb{Z}_3) \rtimes \mathbb{Z}_2$(1) & 6 & No & No & \texttt{GAP}\\\hline
[54, 6] & $(\mathbb{Z}_9 \rtimes \mathbb{Z}_3) \rtimes \mathbb{Z}_2$ & 18 & No & No & \texttt{GAP}\\\hline
[54, 7] & $(\mathbb{Z}_9 \times \mathbb{Z}_3) \rtimes \mathbb{Z}_2$ & 18 & No & Yes & \texttt{GAP}\\\hline
[54, 8] & $(\mathbb{Z}_3^2 \rtimes \mathbb{Z}_3) \rtimes \mathbb{Z}_2$(2) & 6 & No & Yes & \texttt{GAP}\\\hline
[54, 9] & $\mathbb{Z}_{9} \times \mathbb{Z}_3 \times \mathbb{Z}_2$ & 18 & Yes & Yes & T\ref{nilpotent}\\\hline
[54, 10] & $(\mathbb{Z}_3^2 \rtimes \mathbb{Z}_3) \times \mathbb{Z}_2$ & 6 & Yes & Yes & T\ref{nilpotent}\\\hline
[54, 11] & $(\mathbb{Z}_9 \rtimes \mathbb{Z}_3) \times \mathbb{Z}_2$ & 18 & Yes & Yes & T\ref{nilpotent}\\\hline
[54, 12] & $\mathbb{Z}_3^2 \times S_3$ & 6 & No & Yes & C\ref{directarbitrary}\\\hline
[54, 13] & $(\mathbb{Z}_3^2 \rtimes \mathbb{Z}_2) \times \mathbb{Z}_3$ & 6 & No & Yes & C\ref{directarbitrary}\\\hline
[54, 14] & $\mathbb{Z}_3^3 \rtimes \mathbb{Z}_2$ & 6 & No & Yes & \texttt{GAP}\\\hline
[54, 15] & $\mathbb{Z}_3^3 \times \mathbb{Z}_2$ & 6 & Yes & Yes & C\ref{directarbitrary}\\\hline
[56, 1] & $\mathbb{Z}_7 \rtimes \mathbb{Z}_8$ & 56 & No & No & C\ref{ExpCor}\\\hline
[56, 3] & $\mathbb{Z}_7 \rtimes Q_8$ & 28 & No & Yes & C\ref{covering-semi}\\\hline
[56, 4] & $D_{14} \times \mathbb{Z}_4$ & 28 & No & Yes & \texttt{GAP}\\\hline
[56, 5] & $D_{56}$ & 28 & No & Yes & T\ref{EqCovDn}\\\hline
[56, 6] & $(\mathbb{Z}_7 \rtimes \mathbb{Z}_4) \times \mathbb{Z}_4$ & 18 & No & Yes & \texttt{GAP}\\\hline
[56, 7] & $(\mathbb{Z}_{14} \times \mathbb{Z}_2) \rtimes \mathbb{Z}_2$ & 28 & No & Yes & \texttt{GAP}\\\hline
[56, 8] & $\mathbb{Z}_{7} \times \mathbb{Z}_4 \times \mathbb{Z}_2$ & 28 & Yes & Yes & T\ref{nilpotent}\\\hline
[56, 9] & $D_8 \times \mathbb{Z}_7$ & 28 & Yes & Yes & C\ref{directarbitrary}\\\hline
[56, 10] & $Q_8 \times \mathbb{Z}_7$ & 28 & Yes & Yes & C\ref{directarbitrary}\\\hline
[56, 11] & $\mathbb{Z}_2^3 \rtimes \mathbb{Z}_7$ & 14 & No & Yes & \texttt{GAP}\\\hline
[56, 12] & $D_{14} \times \mathbb{Z}_2^2$ & 14 & No & Yes & C\ref{directarbitrary}\\\hline
[56, 13] & $\mathbb{Z}_2^3 \times \mathbb{Z}_7$ & 14 & Yes & Yes & C\ref{directarbitrary}\\\hline
[60, 1] & $(\mathbb{Z}_3 \rtimes \mathbb{Z}_4) \times \mathbb{Z}_5$ & 60 & No & No & C\ref{ExpCor}\\\hline
[60, 2] & $(\mathbb{Z}_5 \rtimes \mathbb{Z}_4) \times \mathbb{Z}_3(1)$ & 60 & No & No & C\ref{ExpCor}\\\hline
[60, 3] & $\mathbb{Z}_{15} \rtimes \mathbb{Z}_4$ & 60 & No & No & C\ref{ExpCor}\\\hline
[60, 5] & $A_5$ & 30 & No & No & P\ref{simpleexp}\\\hline
Group ID & Name(s) & Exponent & Nilpotent? & Equally Coverable? & Reason\\ \hline 
[60, 6] & $(\mathbb{Z}_5 \rtimes \mathbb{Z}_4) \times \mathbb{Z}_3(2)$ & 60 & No & No & C\ref{ExpCor}\\\hline
[60, 7] & $\mathbb{Z}_{15} \rtimes \mathbb{Z}_4$ & 60 & No & No & C\ref{ExpCor}\\\hline
[60, 8] & $D_{10} \times S_3$ & 30 & No & Yes & \texttt{GAP}\\\hline
[60, 9] & $A_4 \times \mathbb{Z}_5$ & 30 & No & No & \texttt{GAP}\\\hline
[60, 10] & $D_{10} \times \mathbb{Z}_6$ & 30 & No & Yes & \texttt{GAP}\\\hline
[60, 11] & $S_3 \times \mathbb{Z}_{10}$ & 30 & No & Yes & \texttt{GAP}\\\hline
[60, 12] & $D_{60}$ & 30 & No & Yes & T\ref{EqCovDn}\\\hline
[60, 13] & $\mathbb{Z}_2^2 \times \mathbb{Z}_{15}$ & 30 & Yes & Yes & C\ref{directarbitrary}\\\hline
\end{longtable}
Table 2: Classification of Remaining Groups Up to Order 60
\end{center}
\vspace{5pt}
\begin{center}
\begin{tabular}{|c|c|c|c|c|c|}\hline
   Group  &  Order & Exponent & Nilpotent? & Equally Coverable? & Reason\\\hline
   $A_5 \cong L_2(4) \cong L_2(5)$  & 60 & 30 & No & No & P\ref{simpleexp}\\\hline
   $A_6 \cong L_2(9) \cong S_4(2)'$ & 360 & 60 & No & No & \texttt{GAP}\\\hline
   $L_2(8) \cong R(3)'$ & 504 & 126 & No & No & \texttt{GAP}\\\hline
   $L_2(11)$ & 660 & 330 & No & No & P\ref{simpleexp}\\\hline
   $L_2(13)$ & 1092 & 546 & No & No & P\ref{simpleexp}\\\hline
   $L_2(17)$ & 2448 & 1224 & No & No & P\ref{simpleexp}\\\hline
   $A_7$ & 2520 & 420 & No & No & \texttt{GAP} \\\hline
   $L_2(19)$ & 3420 & 1710 & No & No & P\ref{simpleexp}\\\hline
   $L_2(16)$ & 4080 & 510 & No & No & \texttt{GAP}\\\hline
   $L_3(3)$ & 5616 & 312 & No & No & \texttt{GAP}\\\hline
   $U_3(3) \cong G_2(2)'$ & 6048 & 168 & No & No & \texttt{GAP}\\\hline
   $L_2(23)$ & 6072 & 3036 & No & No & P\ref{simpleexp}\\\hline
   $L_2(25)$ & 7800 & 780 & No & No & \texttt{GAP}\\\hline
   $M_{11}$ & 7920 & 1320 & No & No & \texttt{GAP}\\\hline
   $L_2(27)$ & 9828 & 546 & No & No & \texttt{GAP}\\\hline
   $L_2(29)$ & 12180 & 6090 & No & No & P\ref{simpleexp}\\\hline
   $L_2(31)$ & 14880 & 7440 & No & No & P\ref{simpleexp}\\\hline
   $A_8 \cong L_4(2)$ & 20160 & 420 & No & No & \texttt{GAP}\\\hline
   $L_3(4)$ & 20160 & 420 & No & No & \texttt{GAP}\\\hline
   $U_4(2) \cong S_4(3)$ & 25920 & 180 & No & No & \texttt{GAP}\\\hline
   $Sz(8)$ & 29120 & 1280 & No & No & \texttt{GAP}\\\hline
   $L_2(32)$ & 32736 & 2046 & No & No & \texttt{GAP}\\\hline
   $U_3(4)$ & 62400 & 780 & No & No & \texttt{GAP}\\\hline
   $M_{12}$ & 96040 & 1320 & No & No & \texttt{GAP}\\\hline
\end{tabular}\vspace{5pt}\\
Table 3: All Finite Simple Groups Of Order Less than 100,000
\end{center}\newpage
Something interesting to note from Table 2 is the fact that all finite simple groups of order less than 100,000 do not have an equal covering. The question now arises whether the following conjecture is true:\vspace{5pt}\\
\textbf{Conjecture:} If $G$ is a finite simple group, then $G$ has no equal covering.

\subsection{\texttt{\texttt{GAP}} Code}
In Table 1, we determined the equal covering status for the non-cyclic groups up to order 60 that we were not able to determine using our theorems alone, and for Table 2 we had performed the same for the first 20 non-cyclic finite simple groups. In many instances, we had to rely on \texttt{\texttt{GAP}} to provide us an insight into whether or not a given group has an equal covering. Before demonstrating the code I used to determine if a group had an equal covering, let us familiarize ourselves with some of the functions that will be used in the code.
\begin{enumerate}
    \item \texttt{DivisorsInt(n)} is a function whose output is the set of positive divisors for the integer \texttt{n}
    \item \texttt{SmallGroup(n,m)} corresponds to the ID of  $m$th group of order $n$ in \texttt{\texttt{GAP}}'s library of groups (primarily applicable to the groups in Table 1)
    \item \texttt{MaximalSubgroups(G)} provides the collection of all maximal subgroups of group $G$
    \item \texttt{Append(set1, set2)} is a function that takes in two sets \texttt{set1} and \texttt{set2} as inputs, and appends the elements of \texttt{set2} to the set of elements of \texttt{set1}
    \item \texttt{ShallowCopy(set)} makes a mutable copy of the set of elements in \texttt{set}. This is useful for wokring with sets that are immutable, that is, sets that cannot be manipulated in any way.
\end{enumerate}
Now that we know the parts of the code that one may not immediately recognize or know, we will now see the general code I used.\vspace{5pt}\\
Let's say for a given group with the ID $[n,m]$ I want to determine whether it has an equal covering. After computing the exponent of $G$, using \texttt{Exponent(G)}, we use the following algorithm for every positive integer $i$ that is 1) a proper divisor of the order of $G$ and 2) is divisible by the exponent of $G$:
\begin{verbatim}
(1) G := SmallGroup(n,m); S := Subgroups(G);
(2) D := DivisorsInt(Order(G)); orders := [];
(3) for d in D do
(4) if d < Order(G) and RemInt(d, Exponent(G)) = 0 then
(5) Append(orders, [d]);
(6) fi;
(7) od;
(8) truthvalues := [];
(9) for d in orders do
(10) union := [];
(11) for s in S do 
(12) if Order(s) = d then 
(13) Append(union, ShallowCopy(Elements(s)));
(14) fi;
(15) do;
(16) Append(truthvalues, [Elements(G) = Set(union)]);
(17) od;
(18) truthvalues;
\end{verbatim}
We can take a closer look at this code to see what it really does. In line:\vspace{5pt}\\
    (1) I designate what group we are dealing with by \texttt{G} and its set of subgroups by \texttt{S}. Note: Most of the groups in Table 2 could not be designated by a Group ID as we did for the groups in Table 1 and instead had to be set using group name functions in \texttt{\texttt{GAP}}. As an example, for $M_{11}$ I had to write \texttt{G := MathieuGroup(11)}.\vspace{5pt}\\
    (2) - (7) I compute the set of integers that are proper divisors of the order of $G$ and are divisible by the exponent of $G$, and call this set \texttt{orders}.\vspace{5pt}\\
    (8)-(17) I determine for each $d$ in \texttt{orders} if the elements of $G$ equals the union of all maximal subgroups in \texttt{S} of order $d$.\vspace{5pt}\\  
If after executing line 18, we produce a set of truth value(s). We determine if a group has an equal covering if at least one \texttt{true} appears. If none appear, we can conclude $G$ has no equal covering.
\subsubsection*{The Case of $M_{12}$}
When running the code shown in the last page for Mathieu Group $M_{12}$, I had run into the issue of \texttt{\texttt{GAP}} not compiling due the large number of subgroups $M_{12}$ possessed (214,871 to be exact). As \texttt{\texttt{GAP}} mentioned, I used the \texttt{ConjugacyClassesSubgroups} function, which returns the set of all conjugacy classes of the group each having a representative. As a work-around, I came up with the following code to determine whether $M_{12}$ possessed an equal covering.
\begin{verbatim}
(1) G := MathieuGroup(12); C := ConjugacyClassesSubgroups(G); classes := [];
(2) for i in [1..Size(C)] do
(3) if RemInt(Order(C[i][1]), Exponent(G)) = 0 and Order(C[i][1]) < Order(G) then
(4) Add(classes, C[i]);
(5) fi;
(6) od;
\end{verbatim}
What I have done is come up with the set of all conjugacy classes of subgroups of $M_{12}$ that contain a subgroup whose order is divisible by the the exponent of $M_{12}$. As a note, since any two subgroups in the same conjugacy class are isomorphic and therefore have the same order, we will attempt to take the set-theoretic union of all proper subgroups satisfying the property mentioned above.\vspace{5pt}\\
After line (5), we must print the set \texttt{classes}, which will tell us the classes that contain proper subgroups whose orders are divisible by the exponent of $M_{12}$. It turns out that \texttt{classes} contains 2 elements, the conjucacy class of two proper subgroups. Next comes code similar to the last page, where we will take the union of all the subgroups in \texttt{classes[1]} and \texttt{classes[2]}(both of which are sets of subgroups of order 7920) and determine whether the union is indeed equal to the set of elements in $M_{12}$.
\begin{verbatim}
(6) c1 := classes[1]; c2 := classes[2]; union:= [];
(7) for s in c1 do
(8) Append(union, ShallowCopy(Elements(s)));
(9) od;
(10) for s in c2 do
(11) Append(union, ShallowCopy(Elements(s)));
(12) od;
(13) Elements(G) = Set(union);
\end{verbatim}
After we execute line 13, we produce \texttt{false} indicating $M_{12}$ has no equal covering.

\end{document}